\documentclass[11pt]{article}
\usepackage[utf8]{inputenc}
\usepackage[T1]{fontenc}
\usepackage{lmodern}
\usepackage{amsmath}
\usepackage{amsfonts}
\usepackage{amsthm}
\usepackage{amssymb}
\usepackage{dsfont}
\usepackage{blindtext}
\usepackage{titlesec}
\usepackage{mathtools}
\usepackage{epigraph}
\usepackage{enumitem}
\usepackage{geometry}
\usepackage[colorlinks=true, linkcolor=red, citecolor=blue]{hyperref}
\usepackage{mathtools}
\usepackage{hyperref}
\usepackage{tabularx}
\usepackage{comment}
\usepackage{bigints}

\usepackage{tikz}
\usetikzlibrary{shapes,positioning,intersections,quotes,patterns,calc}
\usepackage{graphics}
\usepackage{graphicx}
\usetikzlibrary[topaths]

\usepackage{pgfplots}
\pgfplotsset{compat=1.17} 
\pgfkeys{/pgf/declare function={arcsinh(\x) = 1/2*ln(\x+1)-1/2*ln(1-\x);}}

\usepackage{caption}
\usepackage{subcaption}

\usepackage{mathrsfs} 
\usepackage{textcomp}
\usepackage[T1]{fontenc}
\usepackage{newunicodechar}
\numberwithin{equation}{section}

\theoremstyle{plain}
\newtheorem{thm}{Theorem}[section]
\newtheorem{lem}[thm]{Lemma}
\newtheorem{prop}[thm]{Proposition}
\newtheorem{cor}[thm]{Corollary}

\theoremstyle{definition}
\newtheorem{remark}[thm]{Remark}

\DeclareMathOperator\supp{supp}

\title{Whispering gallery modes for a transmission problem}
\author{Spyridon Filippas \footnote{Laboratoire de Math\'ematiques d'Orsay, Universit\'e Paris-Saclay, B\^atiment 307, 91405 Orsay Cedex France, email: spyridon.filippas@universite-paris-saclay.fr.} }
\date{}

\makeatletter
\def\keywords{
    \vspace{1ex}
    \noindent
    \if@twocolumn
      \small{\bf  Keywords}\/---$\!$    \else
      \begin{center}\small\ {\bf Keywords}\end{center}\quotation\small
    \fi}
\def\endkeywords{\vspace{0.6em}\par\if@twocolumn\else\endquotation\fi
    \normalsize\rm}
\makeatother

\begin{document}

\maketitle

	\begin{abstract}
  We construct a specific family of eigenfunctions for a Laplace operator with coefficients having a jump across an interface. These eigenfunctions have an exponential concentration arbitrarily close to the interface, and therefore could be considered as whispering gallery modes. The proof is based on an appropriate Agmon estimate. We deduce as a corollary that the quantitative unique continuation result for waves propagating in singular media proved by the author in~\cite{fQuantJump} is optimal.
\end{abstract}

\begin{keywords}
  \noindent
Agmon estimates, eigenfunctions, transmission problem, wave propagation
\medskip

\noindent
\textbf{2010 Mathematics Subject Classification:}
35B60, 
47F05,       
93B07, 
35P20.  
\end{keywords}

\newcommand{\dom}{\Omega_-}
\newcommand{\domm}{\Omega_+}

\newcommand{\cm}{c_{-}}
\newcommand{\cp}{c_{+}}

\newcommand{\rel}{\operatorname{Re} \lambda}
	\newcommand{\iml}{\operatorname{Im} \lambda}
	\newcommand{\norm}[3]{\left\Vert #1 \right\Vert_{#2}^{#3}}
	\newcommand{\normsurf}[3]{\left\vert #1 \right\vert_{#2}^{#3}}
	\newcommand{\R}{\mathbb{R}}
	\newcommand{\Rp}{\R_+}
	\newcommand{\Rm}{\R_-}
	\newcommand{\lr}{L^2(\R)}
	\newcommand{\lrp}{L^2(\R^+)}
	\newcommand{\lrm}{L^2(\R^-)}

 \newcommand{\M}{\mathcal{M}}
 \newcommand{\W}{\mathcal{W}}

 \newcommand{\p}{-\Delta_c}

\newcommand{\A}{\mathcal{A}}
\newcommand{\1}{\mathds
1}

\newcommand{\tp}{\psi_h }

\newcommand{\re}{\rho_E}

\newcommand{\ph}{\psi_h}

\newcommand{\ve}{\varepsilon}

\newcommand{\vc}{\overline{V}_c}

\section{Introduction}

Let $\mathcal{M}$ be a compact, connected subset of $\R^n$ with smooth boundary $\partial \M$ and $S$ a smooth hypersurface such that we have the following partition: $\text{Int}(\mathcal{M})\backslash S= \dom\cup \domm$ with $\dom\cap \domm= \emptyset$. We consider the operator $\p:=-\textnormal{div} (c \nabla \cdot)$ acting on $\mathcal{M}$ with $c$ being strictly positive and piecewise smooth but having a jump across the interface $S$. In this note we construct, for specific choices of $\dom,\domm$, $S$ and $c$ eigenfunctions of $\p$ which concentrate exponentially near the interface $S$. These maximally vanishing eigenfunctions are sometimes called \textit{whispering gallery modes} (WGM). We show that (see Theorem~\ref{thm for eigen of -delta} or Theorem~\ref{theorem for disk or surfaces} for a more precise statement):

\begin{thm}
\label{thm generic}

There are sets $\M$, $S$ and coefficients $c$ such that there exist sequences $(\lambda_n)_{n \in \mathbb{N}}$ and $(u_n)_{n \in \mathbb{N}}$ with $\lambda_n \rightarrow +\infty$ and $u_n$ satisfying the transmission conditions 
\begin{equation}
\label{trans cond}
({u_n}_{|\dom})_{|S}=({u_n}_{|\domm})_{|S} \textnormal{ and }c_+\partial_\nu({u_n}_{|\dom})_{|S}=c_-\partial_\nu({u_n}_{|\domm})_{|S},
\end{equation}
such that for all $\omega \subset \M$ with $\textnormal{dist}(\overline{\omega},S)>0$ there exist $C, d>0$ with:

\begin{equation}
\label{estimate coeur}
\norm{u_n}{L^2(\omega)}{}\leq Ce^{-d \sqrt{\lambda_n}}, \quad \p u_n=\lambda_n u_n, \quad \norm{u_n}{L^2(\M)}{}=1. 
\end{equation}

\end{thm}

The data $\M, S, c$ considered in the proof of Theorem~\ref{thm generic} have a rotational symmetry and possible sets for $\M$ include a disk or an annulus. The proof of this is based on an \textit{Agmon estimate} (see Section~\ref{section agmon estimate 1d}). Its main advantage is that it is quite simple and it allows to handle with minor modifications different geometries $\M, S$. Our basic toy model will be the case where $\M$ is an annulus (see Figure~\ref{corona} below) but in Section~\ref{other surfaces of revolution} we explain how one can deal with more general surfaces of revolution.

Our main motivation comes from tunneling estimates in control theory (see for instance~\cite[Section 1.2]{LL:18}). Constructing such eigenfunctions allows to saturate certain observability estimates and therefore deduce their optimality. The idea of exhibiting such examples on surfaces of revolution can be traced back to~\cite{Leb:96} and~\cite{Allibert:98}.  

Another motivation for studying WGM comes from optoelectronics. Indeed, the case where $\M$ is a disk and $S$ a smaller circle in its interior can be seen as a toy model for the orthogonal section of an optical fiber. Indeed, in this case $\Omega_-$ can be considered as the core of the optical fiber surrounded by a cladding $\Omega_+$. For the proof of Theorem~\ref{thm generic} we shall assume that $c_-<c_+$ which means that the refractive index in $\dom$ is higher than in $\domm$ and consequently light stays localized in the core by total internal reflection in the boundary between the core and the cladding. See as well the remarks after Theorem~\ref{thm for waves}.  Concerning the applications to optoelectronics, WGM have been studied numerically as well as from a theoretical point of view in~\cite{WhisperingGalleriesAnalysis1, WhisperingGalleriesAnalysis2}. In these references the eigenvalue problem for $\p$ is studied in an unbounded domain, thus becoming a resonance problem. 

In the recent work~\cite{ConcentrationAndNon} the authors obtain concentration and non-concentration properties for the eigenfunctions of $\p$ depending on the regularity of the coefficient $c$. Discontinuity of the coefficient $c$ corresponds to the case of \textit{layered media}. The methods employed and the geometric context are however different than ours since the proofs are not based on Agmon estimates but rather on the explicit form of the Green kernel of the solutions.

\subsection{Optimality of unique continuation results for operators with jumps}

We consider the geometric context described in the beginning of the introduction. 

The following spectral inequality for eigenfunctions of $\p$ is proved in~\cite [Theorem 1.2]{LRR:10} (and generalized in~\cite{le2013carleman}):

\begin{thm}[Theorem 1.2 in~\cite{LRR:10}]
\label{spectral lrr}
Let $(u_j)_{j \in \mathbb{N}}$ be a Hilbert basis of eigenfunctions of the operator $\p$ with
Dirichlet boundary conditions, satisfying the transmission conditions \eqref{trans cond}. Denote by $\lambda_j$ the associated eigenvalues, sorted in an increasing sequence. Then for any $\omega \subset \textnormal{Int}(\M)$ there exists $C>0$ such that one has, for any $a_j \in \mathbb{C}$:
$$
\sum_{\lambda_j \leq \lambda} |a_j|^2\leq C e^{C \sqrt{\lambda}}  \bigintss_{\omega} \left | \sum_{\lambda_j \leq \lambda} a_ju_j(x) \right |^2 dx.
$$    
\end{thm}

The eigenfunctions exhibited in Theorem~\ref{thm generic} prove that the spectral estimate of Theorem~\ref{spectral lrr} above is sharp in general, even for a single eigenfunction.

The discontinuities of the operator $\p$ can be used to describe waves propagating in non-homogeneous media. The coefficient $c$ can be interpreted as the square of the speed of propagation. We consider the following system which describes the evolution of such a wave,

\begin{equation}
\label{system 1}
\begin{cases}
(\partial_t^2 \p )w=0 & \textnormal{in} \:(0,T) \times \dom\cup \domm \\
w_{|S_-}=w_{|S_+} & \textnormal{in}\: (0,T)\times S \\
(c\partial_\nu w)_{|S_-}=(c \partial_\nu w)_{|S_+} & \textnormal{in}\: (0,T)\times S \\
w=0 & \textnormal{in}\: (0,T)\times \partial\mathcal{M} \\
\left(w,\partial_t w \right)_{|t=0}=(w_0,w_1) & \textnormal{in} \: \mathcal{M},
\end{cases}
\end{equation}
where we denote by $\partial_{\nu}$ the outward unit normal vector to $S$ pointing into $\domm$ and by $w_{|S_{\pm}}$ the traces of $w_{|\Omega_{\pm}}$ on $S$. In unique continuation problems one tries to recover the whole wave from a partial observation. In~\cite{fQuantJump} the following quantitative unique continuation result is proved.

\begin{thm}[Theorem 1.3 in~\cite{fQuantJump}]
\label{thm for waves}
For any non empty subset $\omega \subset \M$ there exist $C, T>0$ such that for all $(w_0,w_1) \in H^1_0( \mathcal{M}) \times L^2(\mathcal{M})$ with $(w_0,w_1)\neq(0,0)$ and $w$ solution of~\eqref{system 1} one has:
\begin{align}
\label{fonction log}
 \norm{(w_0,w_1)}{H^1\times L^2}{}&\leq Ce^{C \Lambda}\norm{w}{L^2((0,T)\times \omega)}{}, \nonumber
\end{align}
where $\Lambda= \frac{\norm{(w_0,w_1)}{H^1\times L^2}{}}{\norm{(w_0,w_1)}{L^2 \times H^{-1}}{}}$. 
\end{thm}

Theorem \ref{thm for waves} above generalizes Theorem 1.1 in~\cite{LL:19} where smooth coefficients are considered.
An important aspect of Theorem~\ref{thm for waves} above is that there is no assumption on the sign of the jump of the coefficient $c$. Suppose, to fix ideas, that $c_-<c_+$ are two constants. We interpret $c_-$ and $c_+$ as the square of the speed of propagation of a wave travelling through two isotropic media $\dom$ and $\domm$ with different refractive indices, $n_-$ and $n_+$ respectively (recall that $n_\pm=1/\sqrt{c_\pm}$). Imagine that a wave starts travelling from a region that is inside $\dom$. One has $\sqrt{\frac{c_-}{c_+}}=\frac{n_+}{n_-}$ and therefore the assumption $c_-<c_+$ translates to $n_->n_+$. Then Snell-Descartes law states that when a wave travels from a medium with a higher refractive index to one with a lower refractive index there is a critical angle from which there is \textit{total internal reflection}, that is no refraction at all. At the level of geometric optics, that is to say, in the high frequency regime such a wave stays trapped inside $\dom$. Therefore one expects that, at least at high frequency, no information propagates from $\dom$ to $\domm$, following the laws of geometric optics. However, Theorem~\ref{thm for waves} states that a part of the wave can always be observed from $\domm$ with an intensity at least exponentially small in terms of the typical frequency $\Lambda$ of the wave. 

In this note we show that indeed, in situations where $\frac{c_-}{c_+}$ is sufficiently small depending on the geometric context one can find waves that are exponentially localized, in the high frequency limit, arbitrarily close to the interface $S$. As a consequence, we deduce that the estimate of Theorem~\ref{thm for waves} is, in general, optimal. For a solution $w$ of \eqref{system 1} we define $\Lambda(w):=\frac{\norm{(w(0),\partial_t w(0))}{H^1\times L^2}{}}{\norm{(w(0),\partial_t w (0))}{L^2 \times H^{-1}}{}}$. We have the following corollary of Theorem~\ref{thm generic}:

\begin{thm}[Whispering-gallery waves]
    \label{optimal uc waves}
In the geometric setting of Theorem~\ref{thm generic} there exist solutions $(w_n)_{n \in \mathbb{N}}$ of \eqref{system 1} with 
$
\norm{w_n(0)}{L^2}{}=1
$
and such that for all $\omega \subset \M$ with $\textnormal{dist}(\overline{\omega},S)>0$ there exist $C, d>0$ with:
\begin{equation}
\label{estimate fleur}
\norm{w_n}{L^2((0,T) \times \omega )}{} \leq Ce^{-d \Lambda (w_n)}.    
\end{equation}

\end{thm}

Theorem~\ref{optimal uc waves} is an immediate consequence of Theorem~\ref{thm generic}. Indeed, take $u_n$ as in Theorem~\ref{thm generic}, then
$$
w_n(t,x):=\cos{(\sqrt{\lambda_n}t)}u_n(x).
$$
satisfies~\eqref{system 1} with $\left(w_n,\partial_t w_n \right)_{|t=0}=(u_n,0)$, $\Lambda(w_n)=\sqrt{\lambda_n}+1$ and \eqref{estimate coeur} implies \eqref{estimate fleur}, up to changing the constant $C$.

\bigskip

The plan of the article is as follows. In Section~\ref{domain of the op} we deal with the domain of the operator $\p$. Then in Section~\ref{the corona case section} we give a detailed proof of Theorem~\ref{thm generic} in the case of an annulus based on an Agmon estimate for a 1D semiclassical Schrödinger operator. Finally, in Section~\ref{other surfaces of revolution} we explain how our arguments can be used to include more general surfaces of revolution.

\subsection{Domain and self-adjointness of the operator}
\label{domain of the op}
Let us recall first some basic facts concerning the operator $\p$, its domain and some general spectral properties. Given a function $u=\mathds{1}_{\dom} u_-+\mathds{1}_{\domm}u_+$ with $u_\pm \in C^\infty( \mathcal{M}) $ one has in the distributional sense 
$$
\nabla u = \mathds{1}_{\dom} \nabla u_- +\mathds{1}_{\domm}\nabla u_+ +(u_--u_+)\delta_S \nu,
$$
where $\delta_S$ is the surface measure on $S$ and $\nu$ is the unit normal vector field pointing into $\domm$. We impose then that
\begin{equation}
\label{trans cond general 1}
u_-{_{|S}}=u_+{_{|S}},    
\end{equation}
and the singular term is removed. Similarly, calculating 
$$
\textnormal{div}(c(x)\nabla u),
$$
we see that the condition 
\begin{equation}
\label{trans cond general 2}
\cp \partial_{\nu} {u_{+}}_{|S}=\cm\partial_{\nu} {u_{-}}_{|S}
\end{equation}
combined with \eqref{trans cond general 1} gives the equality
$$
\textnormal{div}(c(x) \nabla u)=\mathds{1}_{\dom}\textnormal{div}(c_- \nabla u_- )+\mathds{1}_{\domm}\textnormal{div}( c_+ \nabla u_+).
$$
We define then $\mathcal{W}$ as the space of functions of the form 
\begin{equation*}
u=\mathds{1}_{\dom} u_-+\mathds{1}_{\domm}u_+, 
\end{equation*}
with $u_\pm \in C^\infty_0(\mathcal{M})$ and such that \eqref{trans cond general 1} and \eqref{trans cond general 2} hold. These conditions are called \textit{transmission conditions} and for $u \in \mathcal{W}$ one has $\p u \in L^2$. With $\W$ as initial domain $\p$ is symmetric and bounded from below. Indeed, writing $(\cdot,\cdot)$ for the inner product in $L^2$ one has for $u \in \mathcal{W}$
$$
(\p u,u)=-\int_{\dom} \text{div}(c_- \nabla u_-)\overline{u}-\int_{\domm}\text{div}(c_+ \nabla u_+)\overline{u}=\int_{\dom}c_- |\nabla u_- |^2 dx+\int_{\domm}c_+ |\nabla u_+ |^2 dx,
$$
where we have used the transmission conditions and an integration by parts.

One can then consider the Friedrichs extension (see for instance~\cite[Chapter 3.2.4]{LewinSpectrale}) of $\p$, which is a self-adjoint extension of $\p$ whose domain is given by:
$$
\mathcal{A}=\{u \in H^1_0(\M) | \p u \in L^2(\M) \}.
$$
Using elliptic regularity arguments (see eg~\cite[Appendix C.2]{LRLR:13}) one can see that in fact 
\begin{equation}
    \label{def of big domain}
    \mathcal{A}=\{u \in H^1_0(\M) \big \vert  u_{|\dom}\in H^2(\dom), \: u_{|\domm} \in H^2(\domm), \:   {c_+\partial_\nu(u_{|\dom}})_{|S}={c_-\partial_\nu(u_{|\domm}})_{|S}\}.
\end{equation}
In the sequel we shall denote by $u_{\pm}$ the restriction $u_{|\Omega_{\pm}}$ of $u \in \A$ on $\dom$ and $\domm$.

The operator $\p $ with domain $\mathcal{A}$ is then positive, self-adjoint and has a compact resolvent. We deduce that its spectrum solely consists of positive eigenvalues $0\leq\lambda_1\leq\lambda_2\leq...$ of finite multiplicity with $\lambda_j \rightarrow +\infty$.

\bigskip

\textbf{Acknowledgements} The author would like to thank C. Laurent and M. Léautaud for discussions, encouragements, and patient guidance.

\section{The case of an annulus}
\label{the corona case section}

We denote by $B_R$ the open ball in $\R^2$ of radius $R$ centered at $0$ and  study the eigenfunction problem for $(\mathcal{A
}, \p)$ in the following geometric context: Let $0<R_0<R_1<R_2$ and set $\M= \overline{B}_{R_2}\backslash B_{R_0} \subset \R^2, \dom=B_{R_1} \backslash \overline{B}_{R_0}, \domm= B_{R_2} \backslash \overline{B}_{R_1}, S=\overline{B}_{R_1} \backslash B_{R_1}$. For simplicity we assume $c$ piecewise constant, that is $c=\mathds{1}_{\dom}c_-+\mathds{1}_{\domm}c_+$, with $0<c_-<c_+$ to be chosen later on.

\begin{figure}
    \centering
    \begin{tikzpicture}
  
   \coordinate (O) at (0,0) ;
   \filldraw[color=black!60, fill=orange!10, ] circle (3);

   \filldraw[color=red!60, fill=orange!10,  very thick] circle (1.5);

  \filldraw [fill=white] circle (0.15);

 \draw[ color=white] circle (0.15);

 \draw(0,0)--(3,0)  node[pos=1,right] {$R_2$} ; 
  \draw(0,0)--(-1.06,-1.06)  node[pos=1,right] {$R_1$} ; 
  \draw(0,0)--(-0.12,0.12)  node[pos=1,left] {$R_0$} ; 

  \draw node at (3, 3)   {$\mathcal{M}$};

 \draw node at (1.3, 1.3)   {$ \color{red} S $};

\end{tikzpicture}

    \caption{The annulus $\M$ with the interface $S$ where the coefficient $c$ jumps.}
    \label{corona}
\end{figure}
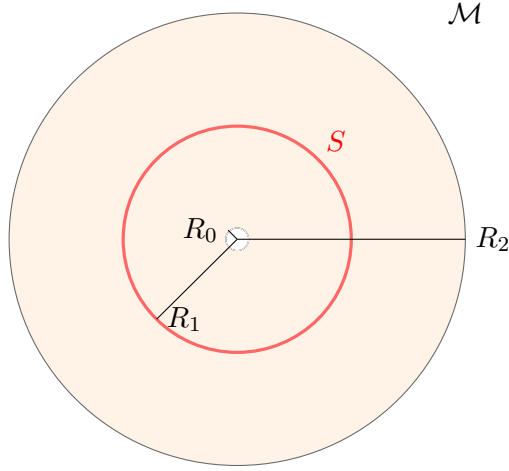

\subsection{Reduction to a semiclassical Schrödinger operator }

We work in polar coordinates, in which the operator $\p$ takes the form 
$$\p=-\frac{1}{r} \partial_r(cr \partial_r)+\frac{c}{r^2}\partial^2_\theta.
$$
For a function $u \in \A$ we have (see Section~\ref{domain of the op}):
$$
\p u(r,\theta)=-\sum_{\pm}\mathds{1}_{\Omega_{\pm}}\left(c_{\pm}\partial^2_r u_\pm +\frac{c_{\pm}}{r}\partial_r u_\pm +\frac{c_{\pm}}{r^2}\partial^2_\theta u_\pm\right),\quad u \in \A,
$$
where we recall that we write $u=\mathds{1}_{\dom}u_-+\mathds{1}_{\domm}u_+$.
We look for solutions of the eigenvalue problem $\p u=\lambda u$ 
 under the form $u(r,\theta)=e^{i n \theta}f(r)$ which yields
\begin{equation*}
    \label{1d equation}
    \sum_{\pm}  \mathds{1}_{\Omega_\pm}\left(-\frac{c_\pm}{n^2}\left(\partial^2_r f_\pm(r)+\frac{1}{r}\partial_r f_\pm(r)\right)+\left(\frac{c_\pm}{r^2}- \frac{\lambda}{n^2}\right)f_\pm(r)\right)=0.
\end{equation*}
We consider large angular momenta, $n \rightarrow + \infty$ and interpret $h:=1/n$ as a semiclassical parameter. Equation \eqref{1d equation} writes then as $P_h f=E_h f$,
where
\begin{equation}
    \label{def of Ph}
    P_h:=-h^2\frac{1}{r} \partial_r(cr \partial_r)+V_c(r), \quad V_c(r):=\frac{c}{r^2}, \quad E_h=h^2 \lambda,
\end{equation}
 and $V_c(r)$ can be seen as the effective potential with a jump discontinuity at $r=R_1$. We define then the continuous extension of $V_c$ to $[R_0,R_1]$:
\begin{equation}
\label{def of extension V}
\vc(r)=
\begin{cases}
    V_c(r),&\quad r\in [R_0,R_1) \\
    \frac{c_-}{R^2_1}, &\quad r=R_1,
\end{cases}
\end{equation}
and we write $E_0:=\frac{c_-}{R^2_1}$ for the infimum of the potential $V_c$.

We now study the eigenvalue problem for the one dimensional semiclassical  Schrödinger operator $P_h$.
We define the space
\begin{align}
    \label{def of Ar}
\A_r=\{f \in H^1_0\left((R_0, R_2)\right) |& f_{|(R_0,R_1)} \in H^2((R_0,R_1)), \:  f_{|(R_1,R_2)} \in H^2((R_1,R_2)), \nonumber \\
& \cp \partial_{r} ({f_{|(R_0,R_1)}})(R_1)=\cm\partial_{r}( {f_{|(R_1,R_2)}})(R_1)\}.
\end{align}
We write as well $f_-:=f_{|(R_0,R_1)}$ and $f_+:=f_{|(R_1,R_2)}$. The space $\A_r$ is such that one has the following two properties:
$$
f \in \A_r \Leftrightarrow f(r)e^{ in \theta} \in \A, \quad n \in \mathbb{Z},
$$
and $P_h$ is self-adjoint with domain $\A_r$, for the same reason as in Section~\ref{domain of the op}. 

In the following proposition we show the existence of eigenfunctions for $P_h$ close to some energy levels. For the proof we start by constructing a rough quasimode and then use the self-adjointness of the operator (see for instance \cite[Lemma 3.6]{LL:18}, \cite[Chapter 12.5]{Zworski:book}).

\begin{prop}[Existence of eigenfunctions for the 1D operator]
\label{existence of eigen 1d}
There exists $h_0$ such that for all $E \in [E_0,V_c(R_0)) $ there exists $C>0$ such that for all $h \in (0,h_0]$, there exist $E_h$ and $\psi_h \in \A_r$ with $$P_h\psi_h=E_h \psi_h, \quad |E_h-E|\leq Ch^{2/3}.$$
\end{prop}
\begin{proof}
The operator $P_h$ with domain $\A_r$ is self-adjoint in $L^2((R_0,R_2),rdr)$. Note that we can estimate indifferently with norms in $L^2((R_0,R_2),rdr)$ or $L^2((R_0,R_2),dr)$ since they are equivalent.

We may write $E=\vc(\re)$ for some $\re \in (R_0,R_1]$ where we recall that $\vc$ is defined in~\eqref{def of extension V}. Consider $\chi \in C^\infty_0((-1,0))$ such that $\chi=1$ in a neighborhood of $-1/2$ and define $f_h(r)=h^{-1/3}\chi\left(h^{-2/3}(r-\rho_E)\right)$. One then has that for $h\leq h_0$ with $h_0$ sufficiently small $f_h(r)\in C^{\infty}_0((R_0,R_1))$. This implies in particular that $f_h \in \A_r$. We estimate now:
\begin{align}
\label{estimate A}
    &\norm{-ch^2\left(\partial^2_r f_h(r)+\frac{1}{r} \partial_r f_h (r)\right)}{L^2}{2} \nonumber \\
    &= \int ch^4 \left(h^{-5/3}\chi^{\prime \prime}\left(h^{-2/3} (r-\rho_E)\right)+\frac{1}{r}h^{-1}\chi^{\prime}\left(h^{{-2/3}}(r-\rho_E)\right)\right)^2 dr \nonumber \\
    &\leq C h^{4/3}.
\end{align}
To estimate the next term we use that $\supp(f_h) \subset (R_0,R_1)$. We use now the fact that $V_c$ is uniformly Lipschitz in $[R_0,R_1]$ which implies that, in the support of $f_h$, one has 
$$
|V_c(r)-E|=|\overline{V}_c(r)-E|\leq C|r-\re|,
$$
and hence
\begin{align}
\label{estimate B}
\norm{\left(V_c(r)-E\right)f_h(r)}{L^2}{2}\leq C \int (r-\rho_E)^2h^{-2/3}\chi^2\left(h^{-2/3}(r-\rho_E)\right)dr\leq Ch^{4/3}.
\end{align}
Putting together \eqref{estimate A} and \eqref{estimate B} yields  
\begin{align}
\label{first estimate for quasimode}
    \norm{\left(P_h-E\right)f_h}{L^2}{}\leq C h^{2/3}.
\end{align}
Noticing finally that
\begin{align*}
    \norm{f_h}{L^2}{2}=\int h^{-2/3}\chi^2\left(h^{-2/3}(r-\rho_E)\right)dr=c_0,
\end{align*}
we can write~\eqref{first estimate for quasimode} as $ \norm{\left(P_h-E\right)f_h}{L^2}{}\leq C h^{2/3}\norm{f_h}{L^2}{}$. If now $E \notin \text{Sp}(P_h)$ one can write
\begin{align*}
  \norm{\left(P_h-E\right)f_h}{L^2}{}&\leq C h^{2/3}\norm{f_h}{L^2}{}=Ch^{2/3}\norm{\left(P_h-E\right)^{-1}\left(P_h-E\right)f_h}{}{} \\
   &\leq Ch^{2/3}\norm{\left(P_h-E\right)^{-1}}{L^2\rightarrow L^2}{} \norm{\left(P_h-E\right)f_h}{L^2}{},
  \end{align*}
which gives
$$
\norm{\big(P_h-E)^{-1}}{L^2\rightarrow L^2}{}\geq C^{-1}h^{-2/3}.
$$
Finally, since $P_h$ is self-adjoint one has $\norm{\big(P_h-E)^{-1}}{L^2\rightarrow L^2}{}=\frac{1}{\text{dist}(\text{Sp}(P_h),E)}$ and consequently
\begin{equation}
    \label{distance from the spectrum}
\text{dist}(\text{Sp}(P_h),E)\leq Ch^{2/3},
\end{equation}
which is trivially true in the case $E \in\text{Sp}(P_h) $ as well. The existence of $E_h, \psi_h$ is then a result of \eqref{distance from the spectrum} and of the fact that the spectrum of $P_h$ consists solely of eigenvalues.
\end{proof}

\subsection{The Agmon estimate}
\label{section agmon estimate 1d}  

We follow~\cite[Chapter 3]{Helffer:booksemiclassic} (see as well~\cite[Chapter 6.B]{DS:book}, \cite{LL:22}). We start by defining the appropriate Agmon distance, which corresponds to a distance to the \textit{classically allowed region} for the potential $V_c$ at the energy level $E$. The classically allowed region $K_E$ is defined as
\begin{equation}
    \label{def of class allowed}
K_E=\{r \in [R_0,R_2]| \: V_c(r) \leq E \}, 
\end{equation}
where we recall that $V_c(r)=\frac{c}{r^2}$ and $c=\mathds{1}_{(R_0,R_1)}c_-+\mathds{1}_{(R_1,R_2)}c_+$. 
The condition we impose on the coefficient $c$ is then (see Figure~\ref{graphs}):
\begin{equation}
    \label{condition on c}
  0<c_-<c_+, \quad  \vc(R_1)<V_c(R_2),
\end{equation}
and as a consequence the function $\Tilde{V}$ defined as $\Tilde{V}=V_c(r)$ for $r\in [R_0,R_2] \backslash\{R_1\}$ and  $\Tilde{V}(R_1)=E_0$ attains its minimum at $r=R_1$. We define the appropriate Agmon distance for $E\geq E_0$ as
\begin{equation}
    \label{Agmon distance}
    d_{A,E}(r)=\underset{y \in K_E}{\text{inf}}\left| \int_{y}^{r} \sqrt{\frac{(V_c(s)-E)_+}{c(s)}}ds \right|,
\end{equation}
where $a_+=\text{max}(a,0)$. 

In the sequel we shall focus on energy levels situated close to the minimum $E_0$. We remark that assumption~\eqref{condition on c} and continuity of $\vc$ in $[R_0,R_1]$ imply that there exists $\eta_0>0$ such that
\begin{equation}
    \label{prop of eta}
   K_{E_0+\eta} \subset [R_0,R_1], \quad  \forall \: 0< \eta < \eta_0.
\end{equation}
For $E\in [E_0, E_0+\eta]$, $\eta\leq\eta_0$ we can define $\re:= V_c^{-1}(E)$ (which is invertible in $[R_0,R_1]$) and obtain that $K_E=[\re,R_1]$ as well as the following explicit expressions:
\begin{align}
    d_{A,E}(r)&=\int_{r}^{\re} \sqrt{\frac{(V_c(s)-E)}{c_-(s)}}ds, \quad \text{for } r \in [R_0,\re], \nonumber\\
    d_{A,E}(r)&=0, \quad  \text{for } r\in[\re,R_1], \\
  \nonumber  d_{A,E}(r)&=\int_{R_1}^{r}\sqrt{\frac{(V_c(s)-E)}{c_+(s)}}ds, \quad  \text{for } r\in [R_1,R_2].
\end{align}
One has in particular that $d_{A,E}$ is $C-$Lipschitz with with $C= \left(\frac{\textnormal{max}\{V_c(s)-E\}}{c_-}\right)^{1/2}$.

The following identity is the key ingredient of the Agmon estimate. Notice that all quantities appearing in the following lemma are well defined. Indeed, since $f \in \A_r$ one has $f \in H^1$ and the left hand side is well defined. 

\begin{lem}
    \label{ipp identity}
    Let $\phi$ be real valued Lipschitz continuous on $[R_0,R_2]$ and $f\in \A_r$. Then one has:
  \begin{multline*}
    \int_{R_0}^{R_2} c h^2\left|\partial_r(e^{\phi/h}f)  \right|^2 rdr-\int_{R_0}^{R_2}c|\partial_r \phi|^2 e^{2 \phi/h}|f|^2 rdr\\
    = -\operatorname{Re}\int_{(R_0,R_1)\cup (R_1,R_2)}  e^{2\phi/h} c h^2\left(\partial^2_r+\frac{1}{r}\partial_r\right)f\cdot\bar{f}rdr.
  \end{multline*}  
\end{lem}
\begin{proof}
We write $f=\mathds{1}_{(R_0,R_1)}f_-+\mathds{1}_{(R_1,R_2)}f_+$, split the integrals and integrate by parts. We have:
\begin{align}
\label{estimate one}
    -\int_{R_0}^{R_1} c_- e^{2\phi/h}\left(\partial^2_r + \frac{1}{r}\partial_r\right)f_-&  \cdot\bar{f}_-rdr =-\int_{R_0}^{R_1} c_-\partial_r(r\partial_r f_-)  e^{2\phi/h} \cdot\bar{f}_-dr  \nonumber\\
    &=\int_{R_0}^{R_1}c_-\partial_r f_-\partial_r(e^{2\phi/h} \bar{f}_-)rdr-c_-R_1\partial_rf_-(R_1)e^{2\phi/h} \bar{f}_-(R_1),
\end{align}
and similarly
\begin{align}
\label{estimate two}
    -\int_{R_1}^{R_2}c_+ e^{2\phi/h}\left(\partial^2_r + \frac{1}{r}\partial_r\right)f_+ &  \cdot\bar{f}_+rdr =-\int_{R_1}^{R_2}c_+\partial_r(r\partial_r f_+)  e^{2\phi/h} \cdot\bar{f}_+dr \nonumber \\
    &=\int_{R_1}^{R_2}c_+ \partial_r f_+\partial_r(e^{2\phi/h} \bar{f}_+)rdr+c_+R_1\partial_rf_+(R_1)e^{2\phi/h} \bar{f}_+(R_1).
\end{align}
Now by definition of the space $\A_r$ in~\eqref{def of Ar} we find that 
$$
c_-R_1\partial_rf_-(R_1)e^{2\phi/h} \bar{f}_-(R_1)=c_+R_1\partial_rf_+(R_1)e^{2\phi/h} \bar{f}_+(R_1).
$$
Adding \eqref{estimate one} and \eqref{estimate two} the boundary terms cancel out and we obtain
\begin{multline*}
   -\int_{(R_0,R_1)\cup (R_1,R_2)} c e^{2\phi/h}h^2\left(\partial^2_r+\frac{1}{r}\partial_r\right)f \cdot\bar{f}rdr\\
   =\int_{R_0}^{R_2} c 2h \partial_r \phi e^{2 \phi/h} \bar{f} \partial_rf rdr+ \int_{R_0}^{R_2} c h^2 e^{2\phi/h}|\partial_r f|^2 rdr,
\end{multline*}
which gives the sought identity after taking real parts.
\end{proof}

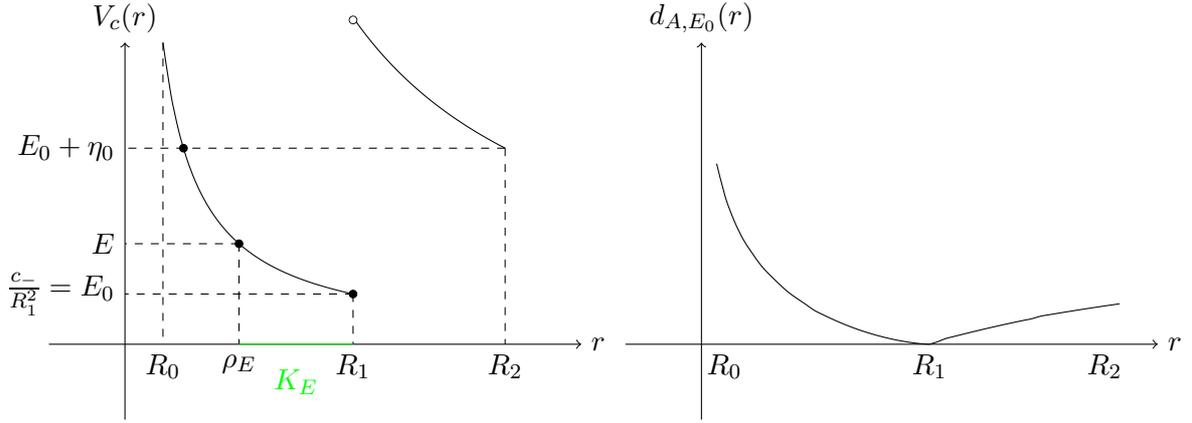
\begin{figure}

    \centering
    \begin{tikzpicture}
    \draw[->] (-1, 0) -- (6, 0) node[right] {$r$};
  \draw[->] (0, -1) -- (0, 4) node[above] {$V_c(r)$};

  \draw[ domain=0.5:3, smooth, variable=\x] plot ({\x}, {2/\x )});
  \draw[ domain=3:5, smooth, variable=\x] plot ({\x}, {13/\x});

  \draw [dashed] (3,0.666) -- (3,0) node [below]{$R_1$}; 
     \filldraw[black] (3,0.666) circle(0.05);

      \filldraw[black] (3,4.3) circle(0.05);
            \filldraw[white] (3,4.3) circle(0.038);

\draw [dashed] (3,0.666) -- (0,0.666) node [left]{$\frac{c_-}{R^2_1}=E_0$}; 
 \draw [dashed] (0.5,4) -- (0.5,0) node [below]{$R_0$}; 
  \draw [dashed] (5,2.6) -- (5,0) node [below]{$R_2$}; 

 \filldraw[black] (0.77,2.6) circle(0.05);

  \draw [dashed] (5,2.6) -- (0, 2.6)node [left]{$E_0+\eta_0$}; 

  \filldraw[black] (1.5,1.333) circle(0.05);

  \draw [dashed] (1.5,1.333) -- (1.5, 0)node [below]{$\re$}; 

   \draw [dashed] (1.5,1.333) -- (0, 1.333)node [left]{$E$}; 

   \draw [green] (1.5, 0) -- (3,0) ; 

   \draw (2.25,-0.5)  node {$\color{green} K_E$};
  
    \end{tikzpicture}
\begin{tikzpicture}

\draw[->] (-1, 0) -- (6, 0) node[right] {$r$};
  \draw[->] (0, -1) -- (0, 4) node[above] {$d_{A,E_0}(r)$};

\draw (3,0) node [below] {$R_1$}; 

\draw  (0.3,0) node [below] {$R_0$}; 

 \draw[ domain=0.2:3, smooth, variable=\x] plot ({\x}, {-sqrt(1-\x*\x/9)+arcsinh(sqrt(1-\x*\x/9))});

 \draw[ domain=3:5.5, smooth, variable=\x] plot ({\x}, {sqrt(1-\x*\x/81)-arcsinh(sqrt(1-\x*\x/81))+0.819938});

   \draw  (5.3,0) node [below]{$R_2$};

\end{tikzpicture}
    
    \caption{Left: The potential $V_c$ with the points of interest. For $E < E_0 +\eta_0=\frac{c_+}{R^2_2}$ the potential is continuous and injective in the classically allowed region $K_E$ which is in green. 
    Right: The graph of the Agmon distance related to the minimal energy level $E_0$. Notice the $C^1$ singularity at $r=R_1$ due to the jump of the coefficient $c$. }
    \label{graphs}

\end{figure}

We can now show the following Agmon estimate. We recall that $\eta_0$ is defined by \eqref{prop of eta}.

\begin{prop}
\label{agmon estimate 1d}
 Let $E \in [E_0, E_0+\frac{\eta_0}{2}]$ and $\epsilon(h)$ with $\epsilon(h)\overset {h \rightarrow 0}{ \longrightarrow }0$. Then for all $\delta>0$, there exist $C$, $h_0$ such that for all $\psi_h$ satisfying
$$
P_h \psi_h=(E+\epsilon(h)) \psi_h,\quad \norm{\psi_h}{L^2}{}=1,
$$
 one has, for $h\leq h_0$:
\begin{equation*}
    \norm{ h \partial_r \left(e^{\frac{d_{A,E}}{h}}  \ph\right)}{L^2((R_0,R_2))}{}+\norm{e^{\frac{d_{A,E}}{h}} \ph}{L^2((R_0,R_2))}{} \leq C e^{\delta/h}.
\end{equation*}

\end{prop}  

\begin{proof}
We recall that $d_{A,E}$ is  Lipschitz continuous and consider the weight
$\phi=(1-\delta)d_{A,E}$ with $0<\delta<1$. Let us write $E_h:=E+\epsilon(h)$. We can then apply the identity of Lemma~\ref{ipp identity} with $\ph$ which solves $P_h \ph=E_h\ph$, or equivalently
$$
-\mathds{1}_{(R_0,R_1)}ch^2\left(\partial^2_r+\frac{1}{r}\partial_r \right)\psi_{h,-}=\mathds{1}_{(R_0,R_1)}(E_h-V_c(r))\psi_{h,-},
$$
and
$$
-\mathds{1}_{(R_1,R_2)}ch^2\left(\partial^2_r+\frac{1}{r}\partial_r \right)\psi_{h,+}=\mathds{1}_{(R_1,R_2)}(E_h-V_c(r))\psi_{h,+}.
$$
We find:
 \begin{multline*}
    \int_{R_0}^{R_2} c h^2\left|\partial_r(e^{\phi/h}\ph)  \right|^2 rdr-\int_{R_0}^{R_2} c|\partial_r \phi|^2 e^{2 \phi/h}|\ph|^2 rdr
    = \int_{R_0}^{R_2}e^{2\phi/h} (E_h-V_c(r))|\ph|^2rdr.
  \end{multline*}

Let us define $I_{\alpha}^+:=\{V-E>\alpha^2\}$, $I_\alpha^-=\{V-E\leq \alpha^2\}$ with $\alpha>0$ small to be chosen. We split the integrals according to $(R_0,R_2)=I_{\alpha}^+ \cup I_{\alpha}^-$ and write the above equality as
\begin{multline}
\label{eq with two split integrals}
  \int_{R_0}^{R_2} c h^2\left|\partial_r(e^{\phi/h}\ph)  \right|^2 rdr+\int_{I_{\alpha}^+}e^{2\phi/h}(V_c(r)-E_h-c|\partial_r \phi|^2)|\ph|^2 rdr \\=- \int_{I_{\alpha}^-}e^{2\phi/h}(V_c(r)-E_h-c|\partial_r \phi|^2)|\ph|^2 rdr.
\end{multline}
To control the integral on $I_{\alpha}^+$ from below we notice that the Agmon distance satisfies the eikonal equation:
\begin{equation}
    \label{eikonal equation}
c|\partial_r d_{A,E}|^2=(V_c(r)-E)_+, \quad \text{in }\mathcal{D}^\prime((R_0,R_2)).
\end{equation}
Hence, taking $h\leq h_0=h_0(\alpha,\delta)$
\begin{align}
\label{estimate from below first integral}
\int_{I_{\alpha}^+}e^{2\phi/h}(V_c(r)-E_h-c|\partial_r \phi|^2)|\ph|^2 rdr& =\int_{I_{\alpha}^+}e^{2\phi/h}\left((V_c(r)-E)(1-(1-\delta)^2)-\epsilon(h)\right)|\ph|^2 rdr  \nonumber \\
& \geq \frac{\alpha^2 \delta}{2} \int_{I_{\alpha}^+}e^{2\phi/h}|\ph|^2 rdr,
\end{align}
where $h_0$ is such that
$
\epsilon(h)\leq \frac{\alpha^2 \delta}{2}
$
for all $h\leq h_0$.
For the integral on $I_{\alpha}^-$ we have
\begin{equation}
    \label{estimate from above second int}
    \left| \int_{I_{\alpha}^-}e^{2\phi/h}(V_c(r)-E_h-c|\partial_r \phi|^2)|\ph|^2 rdr \right| \leq C \int_{I_{\alpha}^-}e^{2\phi/h} |\ph|^2 rdr,
\end{equation}
where $C$ depends on $\text{max }V_c$. Putting \eqref{eq with two split integrals}, \eqref{estimate from below first integral}, \eqref{estimate from above second int} together we find
$$
 \int_{R_0}^{R_2} c h^2\left|\partial_r(e^{\phi/h}\ph)  \right|^2 rdr+\alpha^2 \delta \int_{I_{\alpha}^+}e^{2\phi/h}|\ph|^2 rdr \leq C \int_{I_{\alpha}^-}e^{2\phi/h} |\ph|^2 rdr,
$$
which gives
\begin{equation}
    \label{ineq before agmon 1}
    \int_{R_0}^{R_2} c h^2\left|\partial_r(e^{\phi/h}\ph)  \right|^2 rdr+\alpha^2 \delta \int_{R_0}^{R_2} e^{2\phi/h}|\ph|^2 rdr \leq C \int_{I_{\alpha}^-}e^{2\phi/h} |\ph|^2 rdr.
\end{equation}
We estimate now $\phi$ in $I_{\alpha}^-$. Taking $\alpha^2 \leq \frac{\eta_0}{2}$ we have, with $\re=\vc^{-1}(E)$ 
$$
d_{A,E}(r)=\int_{r}^{\re} \sqrt{\frac{(V_c(s)-E)}{c(s)}}ds ,\quad \text{for }r\in I_{\alpha}^-.
$$
Using that $\vc^{-1}$ is Lipschitz on $[R_0,\re]$ implies that for $r \in I^-_{\alpha}$ one has
$$
|r-\re| \leq C |V_c(r)-V_c(\re)| \leq C \alpha^2.
$$
Choosing then $\alpha_0=\alpha_0(\delta)$ sufficiently small we obtain for $r\in I_{\alpha}^-$ and $ \alpha \leq \alpha_0$:
$$
\phi(r)=(1-\delta)d_{A,E}=(1-\delta)\int_{r}^{\re} \sqrt{\frac{(V_c(s)-E)}{c(s)}}ds \leq C (1-\delta) |r-\re| \leq \Tilde{C} \alpha^2\leq \delta.
$$
which combined with \eqref{ineq before agmon 1} yields, using $R_0>0$,
\begin{equation}
    \label{agmon with phi}
 \int_{R_0}^{R_2}  h^2\left|\partial_r(e^{\phi/h}\ph)  \right|^2 dr+\alpha^2 \delta \int_{R_0}^{R_2} e^{2\phi/h}|\ph|^2 dr \leq C e^{2\delta/h}.
\end{equation}
One needs finally to replace $\phi$ by $d_{A,E}$ in the above estimate. To do this we simply write, with $d_{A,E} \leq M$:
\begin{align*}
  \int_{R_0}^{R_2}h^2 \left | \partial_r \left(e^{\frac{d_{A,E}}{h}}  \ph\right)\right |^2 dr
  & =  \int_{R_0}^{R_2}h^2\left | \partial_r \left(e^{ \frac{ \delta d_{A,E}}{h}} e^{\frac{ \phi}{h}}  \ph\right)\right |^2 dr 
  \\
  &\hspace{-8mm}\leq C \int_{R_0}^{R_2}h^2 e^{ \frac{2\delta d_{A,E}}{h}}\left|\partial_r(e^{\phi/h} \ph)\right|^2 dr
 +C\delta^2 \int_{R_0}^{R_2}|d_{A,E}^\prime|^2 e^{\frac{2 \delta d_{A,E}}{h}} e^{2 \phi/h} |\ph|^2 dr
 \\ &\hspace{-8mm}\leq Ch^2 e^{ \frac{2\delta M}{h}} \int_{R_0}^{R_2} \left|\partial_r(e^{\phi/h}\ph)\right|^2 dr
 +C \delta^2 \norm{V_c-E}{L^\infty}{}e^{ \frac{2\delta M}{h}}\int_{R_0}^{R_2}e^{2 \phi/h} |\ph|^2dr.
 \end{align*}
Combining this together with \eqref{agmon with phi} concludes the proof of Proposition~\ref{agmon estimate 1d}.
\end{proof}

It follows from the Agmon estimate above that the mass of eigenfunctions close to the minimum energy level $E_0$ should be concentrated close to the point $R_1$ where the coefficient $c$ exhibits a jump. This is the following corollary:

\begin{cor}
\label{corollary}
Let $\epsilon(h)$ with $\epsilon(h)\overset {h \rightarrow 0}{ \longrightarrow }0$. There exist $h_0>0$ such that for all $\ph$ satisfying
$$
P_h \psi_h=(E_0+\epsilon(h)) \psi_h,\quad \norm{\psi_h}{L^2}{}=1,\quad \psi_h \in \A_r,
$$
and all $\ve>0$ there exist $C, d>0$ such that
$$
\norm{\ph}{L^2((R_0,R_2)\backslash [R_1-\ve,R_1+\ve])}{}\leq Ce^{-\frac{d}{h }}
$$
\end{cor}

\begin{proof}
We recall that $E_0=V_c(R_1^{-})$ is the infimum of the potential $V_c$. Consider now a solution of $P_h \ph =(E_0+\epsilon(h))\ph$. The associated Agmon distance to the energy level $E_0$ satisfies 
$$
d_{A,E_0}(r)\geq m, \quad \text{for } r \in (R_0,R_2)\backslash [R_1-\ve,R_1+\ve].
$$
The result then follows from the estimate of Proposition~\ref{agmon estimate 1d} by taking $\delta\leq m/2$. 
\end{proof}

\subsection{Back to the two dimensional annulus}

We now put all the pieces together to state our result for the initial operator $\p=-\text{div}(c \nabla \cdot)$ defined on the annulus. Recall that the space $\A$ has been defined in~\eqref{def of big domain}.

\begin{thm}
\label{thm for eigen of -delta}

Consider $\M,S$ as defined in the beginning of Section~\ref{the corona case section}. Suppose that
$$
\frac{c_-}{R_1^2}<\frac{c_+}{R_2^2}.
$$
Then there exist sequences $(\lambda_n)_{n \in \mathbb{N}} \in \R^{\mathbb{N}}, (u_n)_{n \in \mathbb{N}} \in \A^{\mathbb{N}}$ and $E_0>0$ such that for all $\omega \subset \M$ with $\textnormal{dist}(\overline{\omega},S)>0$ there exist $C, d>0$ satisfying:
$$
\p u_n=\lambda_n u_n, \quad \lambda_n \underset{n\rightarrow +\infty}{\sim} E_0 n^2, \quad \norm{u_n}{L^2(\M)}{}=1, \quad \norm{u_n}{L^2(\omega)}{}\leq Ce^{-d n},
$$
for all $n\in \mathbb{N}$.
\end{thm}

\begin{proof}
Let $h=1/n$ and consider $\psi_h \in \A_r$ satisfying 
$$
P_h\ph=(E_0+ O(h^{2/3}))\ph, \quad \norm{\ph}{L^2((R_0,R_2),rdr)}{}=\frac{1}{\sqrt{2 \pi}}.
$$
the existence of such a family $(\ph)_h$ is given by Proposition~\ref{existence of eigen 1d}. We define then
$$
u_n(r,\theta):=e^{i n \theta}\ph(r).
$$
It follows (see the remark after the definition of $\A_r$ in \eqref{def of Ar}) that $u_n \in \A$, $\norm{u_n}{L^2(\M)}{}=1$ and that (see \eqref{1d equation})
$$
\p u_n=n^2\left(E_0+O(n^{-2/3})\right)u_n.
$$
Let finally $\omega \subset \M$ satisfy  $\textnormal{dist}(\overline{\omega},S)>0$. That means that there exists $\ve>0$ with 
\begin{align*}
    \norm{u_n}{L^2( \omega)}{} &\leq \norm{u_n}{L^2(\M \backslash (B(0,R_1+\ve)\backslash B(0,R_1- \ve)))}{}=\norm{\ph}{L^2((R_0,R_2)\backslash [R_1-\ve,R_1+\ve])}{},\\
    &\leq Ce^{-dn},
\end{align*}
thanks to Corollary ~\ref{corollary}.
\end{proof}

\section{The disk and other surfaces of revolution}
\label{other surfaces of revolution}

We considered in Section~\ref{the corona case section} the case where $\M$ is an annulus. The reason we presented the proof for an annulus is that the singularity coming from the change of variables to polar coordinates at $0$ disappears, allowing us to work exclusively in one dimension. Of course, from an heuristic point of view this should not be a problem since we are interested in the behaviour of the eigenfunctions away from zero. However, in the case of the disk it is sometimes simpler to work with the initial operator $\p$ and not with its 1-dimensional analogue~\eqref{def of Ph} since in this case it becomes more intricate to describe its domain $\A_r$ of self-adjointness when $r\in (0,R_2)$.

In this section we briefly explain how the method presented in Section~\ref{the corona case section} can be used to exhibit maximally vanishing eigenfunctions in the case of the disk or even for some surfaces of revolution embedded in $\R^3$ diffeomorphic to a disk. For the geometric description of such manifolds we follow~\cite[Section 3]{LL:18} and~\cite[Section 4]{ LL:18vanish}. 

Let $(\M, g)$ be an embedded 2D submanifold of $\R^3$ having $\mathbb{S}^1$ as an effective isometry group. We denote by $ \mathbb{S}^1 \times \M \ni (\theta, s) \longrightarrow \mathcal{R}_{\theta}s$ the action of $ \mathbb{S}^1$ on $\M$ which satisfies $\mathcal{R}_{\theta} \M=\M$ and we suppose that it has exactly one fixed point denoted by $N \in \M$. We define $L=\textnormal{dist}_g(N, \partial \M)$. Then one can find coordinates 
$$
\M \ni m \longrightarrow \zeta(m)=(s,\theta) \in (0,L] \times \mathbb{S}^1,
$$
such that the metric becomes
$$
(\zeta^{-1})^*g=ds^2+R(s)d\theta^2,
$$
and $R$ is a smooth function $R: [0,L] \rightarrow \R^+_*$ which can be interpreted as the Euclidean distance in $\R^3$ of a point of $\M$ to the symmetry axis. The disk of radius $L$ centered at $0$ corresponds to the case $R(s)=s$. 

In the new coordinates, the Riemannian volume form is $R(s)dsd\theta$ and the Laplace-Beltrami operator is given by 
$$
\Delta_{s, \theta}=\frac{1}{R(s)}\partial_s(R(s)\partial_s)+\frac{1}{R^2(s)} \partial^2_\theta.
$$

Choose now a point $s_0 \in (0,L)$ and define $\dom=\{N\}\cup \zeta^{-1}((0,s_0)\times \mathbb{S}^1)$, $\domm=\zeta^{-1}((s_0,L]\times \mathbb{S}^1)$, $S=\zeta^{-1}(\{s_0\}\times \mathbb{S}^1)$. We consider as well the coefficient $c=c(s)=\mathds{1}_{\dom}c_-+\mathds{1}_{\domm}c_+$, and
$$
\Delta_c:=\frac{1}{R(s)}\partial_s( c(s) R(s)\partial_s)+\frac{c(s)}{R^2(s)} \partial^2_\theta,
$$
which is well defined and self adjoint in the space $\A$, defined in~\eqref{def of big domain} as explained in Section~\ref{domain of the op}. More precisely, for $u \in \A$ one has
$$
\partial_s( c(s) R(s)\partial_s u)=\mathds{1}_{\dom}c_-\partial_s(R(s)\partial_s u_-)+\mathds{1}_{\dom}c_+\partial_s(R(s)\partial_s u_+), \quad \textnormal{in } \mathcal{D}^\prime(\M),
$$
where $u_\pm=u_{|\Omega_{\pm}}$.
The effective potential becomes now $
V_c(s)=\frac{c(s)}{R^2(s)},
$
and we define $V_c(s_0)=\frac{c_-}{R^2(s_0)}$ and $E_0= \textnormal{min }V_c$ . Consider now the operator $\tilde{P}_h$ given by
$$
\tilde{P}_h=-h^2\left( \frac{1}{R(s)}\partial_s(c (s)R(s)\partial_s)\right)+\frac{c(s)}{R^2(s)} \partial^2_\theta .
$$
It can be as shown as in \cite[Section 3.2]{LL:18} that if $\lambda \in \R$ is an eigenvalue of $\Tilde{P}_h$ then there is an eigenfunction of the form $e^{ i n \theta} \psi(s)$ with $n \in \mathbb{Z}$ and $\psi \in L^2((0,R_2),R(s)ds)$. This and the arguments used in the proof of Proposition~\ref{existence of eigen 1d} give us the following:
\begin{prop}
\label{existence of eigen 2d}
For all $n \in \mathbb{N}$ there exists $u_n \in \A$ such that:
$$
\p u_n=n^2\left(E_0+O\left(\frac{1}{n^{2/3}}\right)\right)u_n,\quad \norm{u_n}{L^2(\M)}{}=1, \quad u_n(s,\theta)=e^{i n \theta}\psi_n(s).
$$
    
\end{prop}

We can then repeat the same steps is in Section~\ref{the corona case section}. We define as in~\eqref{Agmon distance} the Agmon distance to the energy level $E$ by

\begin{equation*}
    \label{Agmon distance modified}
    d_{A,E}(s)=\underset{y \in K_E}{\text{inf}}\left| \int_{y}^{r} \sqrt{\frac{(V_c(x)-E)_+}{c(x)}}dx \right|,
\end{equation*}
which satisfies the equation $c|\partial_s d_{A,E}|^2=(V_c(s)-E)_+,$ in $\mathcal{D}^\prime ((0,L))$. Working directly on $\M$ and using the fact that $u_n \in \A$ allows us to obtain the key identity of Lemma~\ref{ipp identity}. We define the Agmon distance on $\M$ as the pullback by $\zeta$ of the Agmon distance defined for $s\in (0,L]$. 

We now assume that $V_c^{-1}(E_0)=\{s_0\}$. That is to say
$$
\frac{1}{R^2(s)}>\frac{1}{R^2(s_0)},\: s\in (0,s_0) \quad \textnormal{and } \frac{c_-}{R^2(s_0)}<\underset{s\in[s_0,L]}{\min} \frac{c_+}{R^2(s)}.
$$
With the key identity at hand we obtain then the Agmon estimate of Proposition~\ref{agmon estimate 2d}. The proof is very similar to the proof of Proposition~\ref{agmon estimate 1d}, but needs some care with respect to the degeneracy at the pole $N$ where the Agmon distance tends to infinity (see Figure~\ref{graphs}). This issue is treated in~\cite[Theorem 3.9]{LL:22} and we omit it.

\begin{prop}
    \label{agmon estimate 2d}
    For all $\delta>0$ and $\epsilon(n)$ with $\epsilon(n)\underset {n \rightarrow \infty}{ \longrightarrow }0$, there exist $C$, $n_0$ such that for all $u_n$ satisfying
$$
\p u_n=n^2(E_0+\epsilon(n)) u_n,\quad \norm{u_n}{L^2}{}=1,\quad u_n \in \A,
$$
one has for $n\geq n_0$:
\begin{equation*}
   \norm{e^{n d_{A,E_0}} u_n }{L^2(\M)}{} \leq C e^{  n \delta}.
\end{equation*}
\end{prop}

 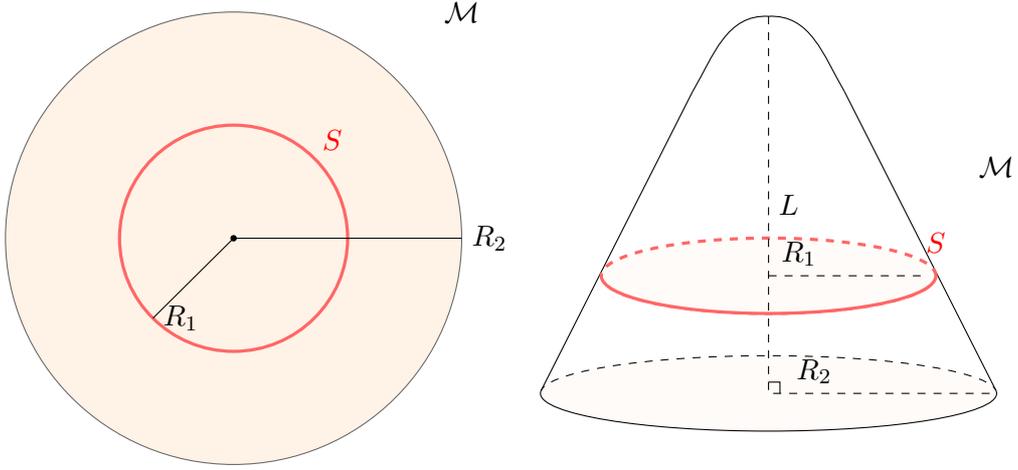
\begin{figure}
    \centering
    \begin{tikzpicture}

   \coordinate (O) at (0,0) ;
   \filldraw[color=black!60, fill=orange!10, ] circle (3);

   \filldraw[color=red!60, fill=orange!10,  very thick] circle (1.5);

 \filldraw circle (1pt);
  \draw(0,0)--(3,0)  node[pos=1,right] {$R_2$} ; 
  \draw(0,0)--(-1.06,-1.06)  node[pos=1,right] {$R_1$} ; 
  \draw node at (3, 3)   {$\mathcal{M}$};

  \draw node at (1.3, 1.3)   {$ \color{red} S $};
 
\end{tikzpicture}
\begin{tikzpicture}

\newcommand{\radiusx}{3}
\newcommand{\radiusy}{.5}
\newcommand{\height}{5}

\newcommand{\radiusz}{2}
\newcommand{\radiusw}{0.5}

\coordinate (a) at (-{\radiusx*sqrt(1-(\radiusy/\height)*(\radiusy/\height))},{\radiusy*(\radiusy/\height)});

\coordinate (b) at ({\radiusx*sqrt(1-(\radiusy/\height)*(\radiusy/\height))},{\radiusy*(\radiusy/\height)});

\draw[fill=orange!3] (a)--(-1,\height-1);

\draw[fill=orange!3] (b)--(1,\height-1);

  \draw (-1,\height-1) to [out=60,in=177] (0,\height);

     \draw (1,\height-1) to [out=120,in=3,] (0,\height);

\fill[orange!3] circle (\radiusx{} and \radiusy);

\begin{scope}
\clip ([xshift=-2mm]a) rectangle ($(b)+(1mm,-2*\radiusy)$);
\draw circle (\radiusx{} and \radiusy);
\end{scope}

\begin{scope}
\clip ([xshift=-2mm]a) rectangle ($(b)+(1mm,2*\radiusy)$);
\draw[dashed] circle (\radiusx{} and \radiusy);
\end{scope}

 \filldraw[dashed, color=red!60, fill=orange!3,  very thick] (0,1.56) circle(\radiusz{}+0.2 and \radiusw);

 \begin{scope}
    \clip (-2.2,1.56) rectangle (2.2,1) ;
    \filldraw[color=red!60, fill=orange!3,  very thick] (0,1.56) circle(\radiusz{}+0.2 and \radiusw);
\end{scope}

\draw node at (2.2, 2)   {$\color{red} S $};

 \draw[dashed] (0,1.56)|-(\radiusz,1.56) node[above,pos=.6]{$R_1$};

\draw[dashed] (0,\height)|-(\radiusx,0) node[right, pos=.25]{$L$} node[above,pos=.6]{$R_2$};

\draw (0,.15)-|(.15,0);

\draw node at (3, 3)   {$\mathcal{M}$};
\end{tikzpicture}

 \caption{A disk and a surface of revolution diffeomorphic to the disk for which Theorem~\ref{theorem for disk or surfaces}} applies under an assumption for the coefficient $c$ similar to the one considered for the annulus.
    \label{disk}
\end{figure}

This gives finally the following result:

\begin{thm}
\label{theorem for disk or surfaces}

Consider $\M,S, c$ as defined in Section~\ref{other surfaces of revolution}. Suppose that $V_c^{-1}(E_0)=\{s_0\}$.

Then there exist sequences $(\lambda_n)_{n \in \mathbb{N}} \in \R^{\mathbb{N}}, (u_n)_{n \in \mathbb{N}} \in \A^{\mathbb{N}}$ such that for all $\omega \subset \M$ with $\textnormal{dist}_g(\overline{\omega},S)>0$ there exist $C, d>0$ satisfying:
$$
\p u_n=\lambda_n u_n, \quad \lambda_n \underset{n\rightarrow +\infty}{\sim} E_0 n^2, \quad \norm{u_n}{L^2(\M)}{}=1, \quad \norm{u_n}{L^2(\omega)}{}\leq Ce^{-d n},
$$
for all $n\in \mathbb{N}$.
\end{thm}

\begin{remark}
If one supposes that the function $s\mapsto R(s)$ is increasing (this is the case if for instance $\M$ is a disk or a cone-like surface) then the assumption $V_c^{-1}(E_0)=\{s_0\}$ is equivalent to
$$
\frac{c_-}{R_1^2}<\frac{c_+}{R_2^2},
$$
where $R_1:=R(s_0)$ and $R_2:=R(L)$. This is the same assumption as~\eqref{condition on c} for the annulus.

\end{remark}

\begin{remark}
Theorem~\ref{theorem for disk or surfaces} and its proof are equally valid in the case where the coefficient $c$ is rotationally invariant and piecewise smooth (not necessarily constant) satisfying $0<c_{\textnormal{min}}\leq c \leq c_{\textnormal{max}}$. Notice that in the case of a disk or an annulus, in order to satisfy $V_c^{-1}(E_0)=\{s_0\}$ with $s_0 \in (0,L)$ the coefficient $c(s)$ has indeed to present a jump discontinuity at $s=s_0$ if one supposes that $c$ is constant in $\Omega_+$ (in terms of physical applications that means that the refractive index of the outer cladding of the optical fiber is constant).
\end{remark}

\small \bibliographystyle{alpha}
\bibliography{bibl}

\begin{thebibliography}{BDDM20}

\bibitem[All98]{Allibert:98}
Brice Allibert.
\newblock Contr\^ole analytique de l'\'equation des ondes et de l'\'equation de
  {S}chr\"odinger sur des surfaces de r\'evolution.
\newblock {\em Comm. Partial Differential Equations}, 23(9-10):1493--1556,
  1998.

\bibitem[BBAD22]{ConcentrationAndNon}
Assia Benabdallah, Matania Ben-Artzi, and Yves Dermenjian.
\newblock Concentration and non-concentration of eigenfunctions of second-order
  elliptic operators in layered media.
\newblock {\em arXiv:2212.05872}, 2022.

\bibitem[BDDM20]{WhisperingGalleriesAnalysis1}
Stéphane Balac, Monique Dauge, Yannick Dumeige, and Zoïs Moitier.
\newblock Mathematical analysis of whispering gallery modes in graded index
  optical micro-disk resonators.
\newblock {\em https://hal.science/hal-02157635/document}, 2020.

\bibitem[BDM21]{WhisperingGalleriesAnalysis2}
Stéphane Balac, Monique Dauge, and Zoïs Moitier.
\newblock Asymptotics for 2d whispering gallery modes in optical micro-disks
  with radially varying index.
\newblock {\em hal-02528150 , version 2}, 2021.

\bibitem[DS99]{DS:book}
Mouez Dimassi and Johannes Sj\"ostrand.
\newblock {\em Spectral asymptotics in the semi-classical limit}, volume 268 of
  {\em London Mathematical Society Lecture Note Series}.
\newblock Cambridge University Press, Cambridge, 1999.

\bibitem[Fil22]{fQuantJump}
Spyridon Filippas.
\newblock Quantitative unique continuation for wave operators with a jump
  discontinuity across an interface and applications to approximate control.
\newblock {\em arXiv preprint arXiv:2210.04634}, 2022.

\bibitem[Hel88]{Helffer:booksemiclassic}
Bernard Helffer.
\newblock {\em Semi-classical analysis for the {S}chr\"odinger operator and
  applications}, volume 1336 of {\em Lecture Notes in Mathematics}.
\newblock Springer-Verlag, Berlin, 1988.

\bibitem[Leb96]{Leb:96}
Gilles Lebeau.
\newblock \'{E}quation des ondes amorties.
\newblock In {\em Algebraic and geometric methods in mathematical physics
  ({K}aciveli, 1993)}, volume~19 of {\em Math. Phys. Stud.}, pages 73--109.
  Kluwer Acad. Publ., Dordrecht, 1996.

\bibitem[Lew18]{LewinSpectrale}
Mathieu Lewin.
\newblock {\em Théorie spectrale et mécanique quantique}.
\newblock Lecture notes. {https://hal.archives-ouvertes.fr/cel-01935749v4},
  2018.

\bibitem[LL19]{LL:19}
Camille Laurent and Matthieu L\'{e}autaud.
\newblock Quantitative unique continuation for operators with partially
  analytic coefficients. {A}pplication to approximate control for waves.
\newblock {\em J. Eur. Math. Soc. (JEMS)}, 21(4):957--1069, 2019.

\bibitem[LL21a]{LL:18}
Camille Laurent and Matthieu L\'{e}autaud.
\newblock Observability of the heat equation, geometric constants in control
  theory, and a conjecture of {L}uc {M}iller.
\newblock {\em Anal. PDE}, 14(2):355--423, 2021.

\bibitem[LL21b]{LL:18vanish}
Camille Laurent and Matthieu L\'{e}autaud.
\newblock On uniform observability of gradient flows in the vanishing viscosity
  limit.
\newblock {\em J. \'{E}c. polytech. Math.}, 8:439--506, 2021.

\bibitem[LL22]{LL:22}
Camille Laurent and Matthieu L\'eautaud.
\newblock Uniform observation of semiclassical {S}chr\"odinger eigenfunctions
  on an interval.
\newblock {\em to appear in Tunisian J. Math.}, 2022.

\bibitem[LR10]{LRR:10}
J\'er\^ome {Le~Rousseau} and Luc Robbiano.
\newblock Carleman estimate for elliptic operators with coefficents with jumps
  at an interface in arbitrary dimension and application to the null
  controllability of linear parabolic equations.
\newblock {\em Arch. Rational Mech. Anal.}, 105:953--990, 2010.

\bibitem[LRL13]{le2013carleman}
J{\'e}r{\^o}me Le~Rousseau and Nicolas Lerner.
\newblock Carleman estimates for anisotropic elliptic operators with jumps at
  an interface.
\newblock {\em Analysis \& PDE}, 6(7):1601--1648, 2013.

\bibitem[LRLR13]{LRLR:13}
J{\'e}r{\^o}me Le~Rousseau, Matthieu L{\'e}autaud, and Luc Robbiano.
\newblock Controllability of a parabolic system with a diffuse interface.
\newblock {\em Journal of the European Mathematical Society}, 15(4):1485--1574,
  2013.

\bibitem[Zwo12]{Zworski:book}
Maciej Zworski.
\newblock {\em Semiclassical analysis}, volume 138 of {\em Graduate Studies in
  Mathematics}.
\newblock American Mathematical Society, Providence, RI, 2012.

\end{thebibliography}

\end{document}